\documentclass[12pt,a4paper]{article}
\usepackage[latin1]{inputenc}
\usepackage{amsmath}
\usepackage{appendix}
\usepackage{multirow}
\usepackage{url}
\usepackage{graphicx}
\usepackage[all]{xy}
\usepackage{amsfonts}
\date{}
\usepackage{amsthm}
\usepackage{amssymb}
\theoremstyle{plain}

\newtheorem{theorem}{Theorem}[section]
\newtheorem{definition}{Definition}[section]

\newtheorem{remark}{Remark}[section]
\newtheorem{corollary}{Corollary}[section]
\newtheorem{example}{Example}[section]
\numberwithin{equation}{section}
\usepackage{fullpage}
\title{On Andrews' partitions with parts separated by parity}
\author{Abdulaziz M. Alanazi  and Darlison Nyirenda}
\begin{document}
\maketitle

\begin{abstract}
In this paper, we present a generalization of one of the theorems in [G. E. Andrews, Partitions with parts separated by parity, \textit{Annals of Combinatorics} \textbf{23}(2019), 241 - 248], and give its bijective proof. Further variations of related partition functions are studied resulting in a number of interesting identities. 
\end{abstract}

\section{Introduction, Definitions, Notation}
Parity in partitions has played a useful role.  A partition of an integer $n > 0$ is a representation $(\lambda_1, \lambda_2, \ldots, \ldots)$ where $\lambda_{i} > \lambda_{i + 1}$ for all $i$ and $\sum\limits_{j \geq 1} \lambda_{j} = n$. The integer $n$ is called the weight of the partition. However when further restrictions are imposed on the parts $\lambda_i$'s, we get unrestricted partition functions. One such is the number of partitions into distinct parts. This means each part in a partition occurs only once. Parity of this partition function is known, and several authors, including Andrews \cite{andrewsparity} have delved into a broader subject, where parity affects parts of partitions.   There are various resources on the theory of integer partitions, and the interested reader is referred to \cite{andrews0}. On this specific subject, one may consult \cite{andrewsparity}, and citations listed in \cite{andrews}. 
\begin{definition}
Consider a partition $\lambda$ of $n$. Suppose $\lambda = (\lambda_1^{m_1}, \lambda_2^{m_2}, \ldots, \lambda_{l}^{m_{ell}})$ where $m_i$ is the multiplicity of $\lambda_i$ and $\lambda_1 > \lambda_2 > \ldots > \lambda_{\ell}$. Define another partition $\lambda^{\prime}$ whose  $j-\text{th}$ part is given by
$$ \lambda_{j}^{\prime} =  \left(\sum\limits_{i=1}^{\ell -j + 1} m_i\right)^{\lambda_{\ell - j + 1} - \lambda_{\ell - j + 2}},\quad \quad \text{where}\,\,\,\lambda_{\ell + 1} : = 0.$$
The partition $\lambda^{\prime}$ is called the conjugate of $\lambda$ and has weight $n$.
\end{definition}
\noindent Given two partitions $\lambda$ and $\mu$, we consider the union $\lambda \cup \mu$ to be the multiset union, and $\lambda + \mu$  is the sum of two partitions obtained via vector addition in which the $i^{\text{th}}$ largest part of $\lambda + \mu$ is equal to the sum of the $i^{\text{th}}$ largest parts in $\lambda$ and $\mu$. In finding the sum $\lambda +  \mu$, the partition with smaller length must have zeros appended to it in order to match in length with the other partition. Similar rules apply to computing $\lambda - \mu$. \\
\noindent Suppose $\mu$ is a subpartition of $\lambda$. We define a new partition $sub(\lambda , \mu)$ to be a partition obtained by deleting  $\mu$ from $\lambda$. For instance $sub((8,7^2, 6^3, 2^3), (7,6^3, 2)) = (8,7,2^{2})$. Further,  $L_{k}(\lambda)$ is the partition obtained by multiplying $k$ to each part of $\lambda$  whose multiplicity is divisible by $k$ and dividing its multiplicity by $k$. On the other hand, $L_{k}^{-1}(\lambda)$ is obtained by dividing by $k$ each part part divisible by $k$ and multiplying its multiplicity by $k$.\\
For $q$-series, we use the following standard notation:\\
$$(a;q)_{n} = \prod_{i = 1}^{n - 1}(1 - aq^{i}),\quad \quad (a;q)_{\infty} = \lim\limits_{n\rightarrow \infty}(a;q)_{n}, \quad \quad (a;q)_{n} = \frac{(a;q)_{\infty}}{(aq^{n};q)_{\infty}}.$$
Some $q$-identities which will be useful  are recalled as follows:
\begin{equation}\label{here0}
\sum\limits_{n = 0}^{\infty}\frac{(a;q)_{n}}{(q;q)_{n}}q^{\frac{n(n + 1)}{2}} = \prod\limits_{n = 1}^{\infty}(1 - aq^{2n - 1})(1 + q^{n}),
\end{equation}
\begin{equation}\label{here1}
\sum\limits_{n = 0}^{\infty}\frac{q^{2n^2}}{(q;q)_{2n}} = \prod\limits_{n = 1}^{\infty}\frac{(1 + q^{8n - 3})(1 + q^{8n - 5})(1 - q^{8n})}{1 - q^{2n}},
\end{equation}
\begin{equation}\label{here2}
\sum_{n = 0}^{\infty}\frac{(a;q)_{n}(b;q)_{n}}{(q;q)_{n}(c;q)_{n}}z^{n} = \frac{(b;q)_{\infty}(az;q)_{\infty}}{(c;q)_{\infty}(z;q)_{\infty}}\sum_{n = 0}^{\infty} \frac{(c/b;q)_{n}(z;q)_{n}}{(q;q)_{n}(az;q)_{n}}b^{n}, \,\,\, |z| < 1, |b| < 1, |q| < 1.
\end{equation}
\noindent For proof of the above identities, see  \cite{igor}, \cite{slater} and \cite{andrews0}, respectively.\\
Euler discovered the following theorem
\begin{theorem}[Euler, \cite{andrews0}]
The number of partitions of $n$ into odd parts is equal to the number of partitions of $n$ into distinct parts.  
\end{theorem}
This theorem has an interesting bijective proof supplied by J .W. L Glaisher (see \cite{glaisher}). We shall denote Glaisher's map by $\phi$.  In fact $\phi$ converts a partition into odd parts into a partition into disctinct parts.

Let $\lambda = (\lambda_{1}^{m_1}, \lambda_{2}^{m_2}, \ldots,\lambda_{r}^{m_r})$ be a partition of $n$ whose parts are odd. Note that the notation for $\lambda$ implies $\lambda_1 > \lambda_{2} > \ldots$ are parts with multiplicities $m_1, m_2, \ldots$, respectively.\\
\noindent Now, write $m_i$'s in $k$-ary expansion, i.e. 
$$ m_i = \sum\limits_{j = 0}^{l_i}a_{ij}2^{j} \,\,\,\,\text{where}\,\,\,\,0\leq a_{ij}\leq 1.$$
We map $\lambda_i^{m_{i}}$ to $\bigcup_{j= 0}^{l_i}(2^{j}\lambda_i)^{a_{ij}}$, where now $2^{j}\lambda_i$ is a part with multiplicity $ a_{ij}$. The image of $\lambda$ which we shall denote by $\phi(\lambda)$, is given by $$\bigcup_{i = 1}^{r} \bigcup_{j = 0}^{l_i} (2^{j}\lambda_{i})^{a_{ij}}.$$
\noindent Clearly, this is a partition of $n$ with distinct parts.\\
\noindent On the other hand, assume that $\mu = (\mu_1^{f_1}, \mu_{2}^{f_2}, \ldots)$ is a partition of $n$ into ditinct parts. Write $\mu_i = 2^{r_i}\cdot a_i$ where $2\nmid a_i$ and then map $\mu_i^{f_i}$ to $(a_i)^{2^{r_i}f_i}$ for each $i$, where now $a_i$ is a part with multiplicity $2^{r_i}f_i$. The inverse of $\phi$ is then given by  
$$\phi^{-1}(\mu) = \bigcup_{i \geq 1}(a_i)^{2^{r_i}f_i}.$$ 
In the resulting partition, it is also clear that the parts are odd. \\
We also recall the following notation from \cite{andrews}.\\
$p_{eu}^{od}(n)$: the number of partitions of $n$ in which odd parts are distinct and greater than even parts.\\
\noindent $\mathcal{O}_{d}(n)$: the number of partitions of $n$ in which the odd parts are distinct and each odd integer smaller than the largest odd part must appear as a part. Theorem 2 of \cite{andrews} is restated below.
\begin{theorem}[Andrews, \cite{andrews}]\label{and0}
For $n\geq 0$, we have
$$ p_{eu}^{od}(n) = \mathcal{O}_{d}(n).$$
\end{theorem}
\noindent In this paper, we generalise Theorem \ref{and0} and look at various variations. 
\section{A generalisation of Theorem \ref{and0}}
Define $D(n,p,r)$ to be the number of partitions of $n$ in which parts are congruent to $0, r \pmod{p}$, and each part congruent to $ r\pmod{p}$ is distinct and greater than parts congruent to $0 \pmod{p}$. Our theorem is stated below.
\begin{theorem}\label{and1}
Let $O(n,p,r)$ be the number of partitions of $n$ in which parts are congruent to $0, r \pmod{p}$, parts $\equiv r \pmod{p}$ are distinct, and  each integer congruent to $r \pmod{p}$ smaller than the largest part that is congruent to $r\pmod{p}$  must appear as a part. Then 
$$ D(n,p,r) = O(n,p,r).$$ 
\end{theorem} 
\noindent Setting $p = 2, r = 1$  in Theorem \ref{and1} gives rise to Theorem \ref{and0}.
We give a desired bijective proof.\\
\noindent Let $\lambda$ be enumerated by $O(n,p,r)$. We have the decomposition $\lambda = (\lambda_1, \lambda_2)$ where
$\lambda_1$ is the subpartition of $\lambda$ whose parts are $\equiv r \pmod{p}$, and $\lambda_2$ is the subpartition of $\lambda$ whose parts are congruent to $0 \pmod{p}$. Then the image is given by $\lambda_1 + \lambda_2$, i.e.
$$\lambda \mapsto \lambda_1 + \lambda_2$$
\noindent The inverse of the bijection is given as follows:\\
\noindent Let $\mu$ be a partition enumerated by $D(n,p,r)$. Then decompose $\mu$ as
$\mu = (\mu_1, \mu_2)$ where $\mu_1$ is the subpartition with parts congruent to $r \pmod{p}$ and $\mu_2$ is the subpartition with parts congruent to $0 \pmod{p}$. Construct $\mu_3$ as 
$$  \mu_3 = (p\ell(\mu_1) -  p + r, p\ell(\mu_1) -  2p + r, p\ell(\mu_1) - 3p + r, \ldots, r + 2p, r + p, r)$$
where $\ell(\mu_1)$ is the number of parts in $\mu_1$.\\
\noindent Then the image of $\mu$ is given by
$$ \mu \mapsto \mu_2 \cup  \mu_3 \cup  \left[\mu_1 - \mu_3 \right]. $$ 
\begin{example}
Consider $p = 4$, $r = 1$ and an $O(186,4,1)$-partition \\
$\lambda = (32, 32, 21, 17, 16, 13, 9,8,8,8,8,5,4, 4,4,1)$.
\end{example}
\noindent  By our mapping, $\lambda$ decomposes as follows:
$$ \lambda = \left((21,17,13,9,5,1), (32,32,16,8,8,8,8,4,4,4)\right).$$
The image is then given by
$$ (21,17,13,9,5,1,0,0,0,0) + (32,32,16,8,8,8,8,4,4,4)$$
\noindent(we append zeros to the subpartition with smaller length), and addition is componentwise in the order demonstrated. Thus
$$\lambda \mapsto (53, 49, 29, 17, 13, 9, 8, 4,4,4)$$
which is a partition enumerated by $D(190,4,1)$. \\\\
To invert the process, starting with $\mu = (53, 49, 29, 17, 13, 9,8,4,4,4)$, enumerated by $D(190,4,1)$, we have 
the decomposition $\mu = (\mu_1, \mu_2) = \left((53,49, 29, 17,13, 9), (8,4,4,4)\right)$ where $\mu_1 = (53,49, 29, 17, 9)$ and $\mu_2 = (8,4,4,4)$.\\
\noindent Note that $\ell(\mu_1) = 5$ so that $\mu_3 = (17,13,9,5,1)$. Hence, the image is
\begin{align*}
 \mu_2 \cup  \mu_3 \cup  \left[\mu_1 - \mu_3 \right] & = (8,4,4,4) \cup  (21, 17, 13, 9, 5, 1) \cup  \left[(53,49,29,17, 13,9) - (21, 17,13,9,5,1) \right]\\
                                                     & = (8,4,4,4) \cup (21, 17, 13, 9, 5, 1) \cup (32,32,16,8,8,8) \\
                                                     & = (32,32,21,17, 16, 13, 9, 8,8,8,8, 5,4,4,4,1),
\end{align*}
which is enumerated by $O(186,4,1)$ and the $\lambda$ we started with.
\begin{corollary}
The number of partitions of $n$ in which all parts $\not\equiv 0 \pmod{p}$ form an arithmetic progression with common difference $p$  and the smallest part is less than $p$ equals the number of partitions of $n$ in which parts $\not\equiv 0 \pmod{p}$ are distinct, have the same residue modulo $p$ and are greater than parts $\equiv 0 \pmod{p}$.
\end{corollary}
\begin{proof}
By Theorem \ref{and1}, we have $\sum\limits_{r = 1}^{p - 1}O(n,p,r) = \sum\limits_{r = 1}^{p - 1}D(n,p,r)$.
\end{proof}


\section{Related variations}
\noindent In Theorem \ref{and0}, if we reverse the roles of odd and even parts by letting any positive even integer less than the largest even part appear as a part and each odd part be greater than the largest even part, we obtain the following theorem.
\begin{theorem}\label{and2}
Let $r = 1, 3$ and $A(n,r)$ denote the number of partitions of $n$ in which each even integer less than the largest even part appears as a part and the smallest odd part is at least $r$ + the largest even part. Then $A(n,r)$ is equal to the number of partitions of $n$  with parts $\equiv r,2 \pmod{4}$.
\end{theorem}
\begin{proof}
\begin{align*}
\sum_{n = 0}^{\infty}A(n,r)q^{n} & = \frac{1}{\prod_{j = 0}^{\infty}(1 - q^{2j + r})} + \sum_{n = 1}^{\infty}\frac{q^{n(n+ 1)}}{(q^{2};q^{2})_{n}}\frac{1}{\prod_{j = n}^{\infty}(1 - q^{2j + r})}
\end{align*}
\begin{align*}
                                   & = \sum_{n = 0}^{\infty}\frac{q^{n(n+ 1)}}{(q^{2};q^{2})_{n}}\frac{1}{\prod_{j = n}^{\infty}(1 - q^{2j + r})}\\
                                   & = \sum_{n = 0}^{\infty}\frac{q^{n(n+ 1)}}{(q^{2};q^{2})_{n}}\frac{1}{(q^{2n + r};q^{2})_{\infty}}\\
                                   & = \sum_{n = 0}^{\infty}\frac{q^{n(n+ 1)}}{(q^{2};q^{2})_{n}}\frac{(q^{r};q^{2})_{n}}{(q;q^{2})_{\infty}}\\
                                   & = \frac{1}{(q;q^{2})_{\infty}}\sum_{n = 0}^{\infty}\frac{q^{n(n+ 1)}}{(q^{2};q^{2})_{n}} (q^{r};q^{2})_{n}\\
                                   & = \frac{1}{(q;q^2)_{\infty}}\prod_{n = 1}^{\infty}(1 - q^{4n - r})(1 + q^{2n})\,\,(\text{by}\,\, \eqref{here0})\\  
                                   & = \frac{\prod_{n = 1}^{\infty}(1 - q^{4n - r})(1 + q^{2n})(1 - q^{2n})}{(q;q^{2})_{\infty}(q^{2};q^{2})_{\infty}}\\
                                   & = \prod_{n = 1}^{\infty}\frac{(1 - q^{4n - r})(1 - q^{4n})}{1 - q^{n}}\\
                                   & = \prod_{n = 1}^{\infty}\frac{1}{(1 - q^{4n + r})(1 - q^{4n + 2})}.
\end{align*}
\end{proof}
A bijection is given in Section \ref{sec}.
\begin{theorem}\label{och}
Let $C(n)$ be the number of partitions where if $2j$ occurs, then all even integers less than $2j$ occur as parts and any part greater than $2j$ is odd. Then $C(n) \equiv 1\pmod{2}$ if and only if $n = \frac{j(j + 1)}{2}$  for some $j\geq 0$ 
\end{theorem}
\begin{proof}
\begin{align*}
\sum_{n = 0}^{\infty}C(n)q^{n} & = \sum_{n = 0}^{\infty}\frac{q^{2 + 4 + 6 + \ldots + 2n}}{(1 - q)(1 - q^{2})\ldots (1 - q^{2n})(1 - q^{2n + 1})(1 - q^{2n + 3})\ldots}\\
                               & = \sum_{n = 0}^{\infty}\frac{q^{n^{2} + n}}{(q;q)_{2n}(q^{2n + 1};q^{2})_{\infty}}\\
                               & = \sum_{n = 0}^{\infty}\frac{q^{n^{2} + n}}{(q;q)_{2n}}\frac{(q;q^{2})_{n}}{(q;q^{2})_{\infty}} \\
                              & = \frac{1}{(q;q^{2})_{\infty}}\sum_{n = 0}^{\infty}\frac{q^{n^{2} + n}}{(q^{2};q^{2})_{n}} \\
                              & = \frac{1}{(q;q^{2})_{\infty}}\prod_{n = 1}^{\infty}(1 + q^{2n}) \,\, (\text{by}\,\,\eqref{here0},\,\,a = 0, q:= q^{2})
\end{align*}
\begin{align*}
                              & =  \prod_{n = 1}^{\infty}\frac{1 - q^{4n}}{1 - q^{n}}\\
                              & \equiv \prod_{n = 1}^{\infty}(1 - q^{n})^{3} \pmod{2} \\
                              & \equiv \sum\limits_{n = 0}^{\infty} q^{n(n + 1)/2} \pmod{2}          
\end{align*}

\end{proof}
\begin{remark}
It is clearly observable from line 6 of the proof that $C(n)$ is equal to the number of partitions of $n$ into parts not divisible by 4. To prove this partition identity combinatorially, decompose $\lambda \in C(n)$ into $(\lambda_o, \lambda_{e})$ where $\lambda_o$ is the subpartition consisting of odd parts, and $\lambda_e$ is the subpartition consisting of even parts.
Then compute $\phi(\lambda_o)$ and conjugate $\lambda_e$. Split each part of $\lambda_e^{\prime}$ into two identical parts, obtaining $\mu$. Then $$ \phi(\phi(\lambda_o) \cup \mu)$$ is a partition in which parts are not divisible by 4. This transformation is  invertible.
\end{remark}
\subsection{Bijective Proof of Theorem \ref{and2}}\label{sec}
Let  $\mu$ be a partition enumerated by $A(n,r)$. Execute the following steps:\\
\begin{enumerate}
\item[1.] Conjugate $\mu$, obtaining $\mu^{\prime}$.
\item[2.] If $\mu^{\prime}$ has no part with odd multiplicity, set $\bar{\alpha}:= \mu^{\prime}$ and go to step 4. Otherwise, decompose $\mu^{\prime}=(\alpha,\beta)$ where
 
$\beta$ is the subpartition of $\mu^{\prime}$ consisting of all parts less than or equal to largest part that has odd multiplicity. and $\alpha$ is the subpartition of $sub(\mu^{\prime}, \beta)$. Recall that $\beta$ can be written as 
$$\beta = \langle \beta_{1}^{m_{\beta_{1}}(\beta)} \beta_{2}^{m_{\beta_{2}}(\beta)}\ldots \beta_{m}^{m_{\beta_{m}}(\beta)} \rangle $$
where $\beta_{1} > \beta_{2} > \ldots > \beta_{m}$. We use this notation of $\beta$ in the next step.
\item[3.] 
\begin{enumerate}
\item[a.] If $m_{\beta_{1}}(\beta) \not\equiv r \pmod{4}$, then update $\alpha$ and $\beta$  as follows:
 $$\beta := sub(\beta, \langle \beta_{1}^{2}\rangle)\,\,\,\,,\alpha: = \alpha \cup \langle \beta_{1}^{2}\rangle.$$  
\item[b.]  For $j = 2,3,\ldots, m$, if $m_{\beta_{j}}(\beta)  \not\equiv 0 \pmod{4}$, then update $\alpha$ and $\beta$ as follows:   $$\beta := sub(\beta,  \langle \beta_{j}^{2}\rangle) \,\,\,,\alpha: = \alpha \cup \langle \beta_{j}^{2}\rangle.$$
\end{enumerate}
\noindent Now call the new updated $\alpha$ and $\beta$, $\bar{\alpha}$ and $\bar{\beta}$, respectively. Observe that $\mu^{\prime} = \bar{\alpha} \cup \bar{\beta}$.
 \item[4.] Compute $$ \gamma = L_{2}(\phi(L_{2}(\bar{\alpha}) )).$$
\end{enumerate}
Note that $$\lambda = \bar{\beta}^{\prime}\cup \gamma $$
is a partition into parts $\equiv r,2 \pmod{4}$.\\\\
\noindent Before giving the inverse mapping, let us look at an example.\\\\
\underline{Example}\\\\
Let $r = 1$ with $\mu = \langle 23^{2}17^{1}11^{3}8^{1}6^{3}4^{1}2^{4}\rangle \in A(124,1)$. Then $\mu^{\prime} = \langle 15^{2}11^{2}10^{2}7^{2}6^{3}3^{6}2^{6}\rangle$. Thus $\alpha = 15^{2}11^{2}10^{2}7^{2}$ and $\beta = 6^{3}3^{6}2^{6}$.
Updating $\alpha$ and $\beta$ yields: $\bar{\alpha} = \langle 15^{2}11^{2}10^{2}7^{2}6^{2}3^{2}2^{2}\rangle$ and $\bar{\beta} = \langle 6^{1}3^{4}2^{4}\rangle$. \\
Now we have $L_{2}(\bar{\alpha}) = \langle 30,22,20,14,12,6,4\rangle$ so that $ \phi (L_{2}(\bar{\alpha})) = \langle 15^{2}11^{2}7^{2}5^{4}3^{6}  1^{4} \rangle $. Thus
$$ \gamma = L_{2}(\phi (L_{2}(\bar{\alpha})) ) = \langle 30^{1}22^{1}14^{1}10^{2}6^{3}2^{2} \rangle. $$
Since $\bar{\beta}^{\prime} = \langle 9^{2}5^{1}1^{3} \rangle$, the image  is 
$$\lambda = \bar{\beta}^{\prime}\cup  \gamma = \langle 30^{1}22^{1}14^{1}10^{2}9^{2}6^{3}5^{1}2^{2}1^{3}\rangle.$$
\underline{\textit{The inverse}}\\\\
Let $\lambda$ be a partition of $n$ into parts $\equiv r,2 \pmod{4}$. Decompose $\lambda$ as follows 
$\lambda = \lambda_{1}\cup \lambda_{r})$ where $\lambda_{1}$ is the subpartition of $\lambda$ with  parts $\equiv 2 \pmod{4}$ and $\lambda_{r}$ is the subpartition with parts $\equiv r \pmod{4}$.  Compute

$$ h =  L_{2}^{-1}(\phi^{-1}(L_{2}^{-1}(\lambda_{1}))).$$
\noindent Then $$\mu = \lambda_{r} \, + \, h^{\prime}$$ is a partition in $A(n,r)$. \\\\
\noindent For instance, consider $\lambda = \langle 30^{1}22^{1}14^{1}10^{2}9^{2}6^{3}5^{1}2^{2}1^{3}\rangle$ in the example above ($r = 1$). Then $\lambda_{1} = \langle 30^{1}22^{1}14^{1}10^{2}6^{3}2^{2}\rangle$ and $\lambda_{r} = \langle 9^{2}5^{1}1^{3}\rangle$.\\
Now $L_{2}^{-1}(\lambda_{1})= \langle 15^{2}11^{2}7^{2}5^{4}3^{6}1^{4}\rangle $  so that $\phi^{-1}(L_{2}^{-1}(\lambda_{1})) = \langle 30,22,20,14,12,6,4\rangle $.\\
Thus $h = L_{2}^{-1}(\langle 30,22,20,14,12,6,4\rangle) = \langle 15^{2}11^{2}10^{2}7^{2}6^{2}3^{2}2^{2}\rangle $  and that $h^{\prime} = \langle 14^{2}12^{1}10^{3}8^{1}6^{3}4^{1}2^{4} \rangle$. Hence
 $$\mu = \langle 9^{2}5^{1}1^{3}\rangle + \langle 14^{2}12^{1}10^{3}8^{1}6^{3}4^{1}2^{4} \rangle = \langle 23^{2}17^{1}11^{3}8^{1}6^{3}4^{1}2^{4}\rangle.$$

\begin{theorem}
Let $B(n)$ be the number of partitions of $n$ in which either a) all parts are even and distinct or  b) 1 must appear and odd parts appear without gaps, even parts are distinct and each is greater than or equal to 3 + the largest odd part. Denote by $B_{e}(n)$ (resp. $B_{o}(n)$), the number of $B(n)$-partitions with an even (resp. odd) number of even parts. Then
$$ B_{e}(n) - B_{o}(n) = 
\begin{cases}
1 & \text{if}\,\,n = m(4m \pm 1), m \geq 0\\
0 & \text{otherwise}.
\end{cases}
$$
\end{theorem}
\noindent Note that the generating function for the sequence $B(0), B(1), B(2), \ldots$ is $$ \frac{1}{(q^{2};q^{2})_{\infty}} + \sum_{n = 1}^{\infty}\frac{q^{2(1 + 3 + 5 + \ldots + 2n - 1)}}{(q;q^{2})_{n}}(-q^{2n + 2};q^{2})_{\infty}$$
\noindent Hence
\begin{align*}
\sum_{n = 0 }^{\infty}(B_{e}(n) - B_{o}(n))q^{n} & = (q^{2};q^{2})_{\infty} + \sum_{n = 1}^{\infty}\frac{q^{2(1 + 3 + 5 + \ldots + 2n - 1)}}{(q;q^{2})_{n}}(q^{2n + 2};q^{2})_{\infty}\\
                                                 & = \sum_{n = 0}^{\infty}\frac{q^{2n^{2}}}{(q;q^{2})_{n}}\frac{(q^{2};q^{2})_{\infty}}{(q^2;q^2)_{n}}\\
                                                              & = (q^{2};q^{2})_{\infty}\sum_{n = 0}^{\infty}\frac{q^{2n^{2}}}{(q;q)_{2n}}\\
                                                 & = \prod_{n = 1}^{\infty} (1 + q^{8n - 3})(1 + q^{8n - 5})(1 - q^{8n})\,\,\text{by}\,\,\eqref{here1}\\
                                                & = \sum_{n = -\infty}^{\infty}q^{4n^{2} + n},
                                               \end{align*}
and the result follows.
\begin{corollary}
For all $n\geq 0$, $B(n)$ is odd if and only if $n = m(4m \pm 1)$ for some integer $m \geq 0$.
\end{corollary}
\noindent Finally, consider the partition function; \\
$\tau(n)$: the number of partitions of $n$ in which even parts  are distinct or if an even part is repeated, it is the smallest and occurs exactly twice and all other even parts are distinct.\\
\noindent Let $\tau_{e}(n)$ (resp. $\tau_{o}(n)$) denote the number of $\tau(n)$-partitions with an even (resp. odd) number of distincteven parts. Then the following identity follows:
\begin{theorem}\label{yanu}
For all non-negative integers $n$, we have 
$$\tau_{e}(n) - \tau_{o}(n) = 
\begin{cases}
1 ,& \text{if}\,\, n = 3m, 3m + 1, m\geq 0 \\
0 ,& \text{otherwise}
\end{cases}
$$
where $\tau_{e}(0) - \tau_{o}(0): = 1$.
\end{theorem}
\begin{proof}
\noindent Note that 
\begin{align*}
 \sum\limits_{n = 0}^{\infty}\tau(n)q^{n} & = \frac{(-q^{2};q^{2})_{\infty}}{(q;q^{\infty})_{\infty}} + \sum\limits_{n = 1}^{\infty}q^{2n + 2n}(-q^{2n + 2};q^{2})_{\infty} (q^{2n + 1};q^{2})_{\infty}^{-1} \\
                                         & = \sum_{n = 0}^{\infty}q^{2n + 2n}(-q^{2n + 2};q^{2})_{\infty} (q^{2n + 1};q^{2})_{\infty}^{-1}
 \end{align*}
so that 
\begin{align*}
\sum_{n = 0}^{\infty}(\tau_{e}(n) - \tau_{o}(n))q^{n} & = \sum_{n = 0}^{\infty}q^{2n + 2n}(q^{2n + 2};q^{2})_{\infty} (q^{2n + 1};q^{2})_{\infty}^{-1}\\
                                                      & = \frac{(q^{2};q^{2})_{\infty}}{(q;q^{2})_{\infty}}\sum_{n = 0}^{\infty}\frac{q^{4n}}{(q^{2};q^{2})_{n}}(q;q^{2})_{n}\\
                                                      & = \frac{(q^{2};q^{2})_{\infty}}{(q;q^{2})_{\infty}} \frac{(q;q^{2})_{\infty}}{(q^{4};q^{2})_{\infty}} \sum_{n = 0}^{\infty} \frac{q^{n}}{(q^{2};q^{2})_{n}}(q^{4};q^{2})_{n}\\
                                                      &\hspace{3mm}\,\,\left( \text{by}\,\,\eqref{here2}, a=c=0, b= q, t = q^{4}\right). \\
                                                      & = (1 - q^{2})\sum_{n = 0}^{\infty} \frac{q^{n}}{(q^{2};q^{2})_{n}}(q^{4};q^{2})_{n} \\
                                                      & = (1 - q^{2})\sum_{n = 0}^{\infty} \frac{(1 - q^{2n + 2})q^{n}}{1 - q^{2}}\\
                                                      & = \sum_{n = 0}^{\infty}(1 - q^{2n + 2})q^{n}\\
                                                      & = \sum_{n = 0}^{\infty} q^{n} - \sum_{n = 0}^{\infty} q^{3n + 2}\\
                                                      & = \sum_{n = 0}^{\infty} q^{3n} + \sum_{n = 0}^{\infty}q^{3n + 1}.
\end{align*}
\end{proof} 
\begin{example}
Consider $n = 8$.
\end{example}
\noindent The $\tau(8)$-partitions are:
$$(8), (7,1), (6,2), (5,3), (5,1,1,1), (4,4), (4,2,1,1), (3,3,1,1), (3,1,1,1,1,1), $$
$$(6,1,1),(5,2,1), (4,3,1), (4,2,2), (4,1,1,1,1), (3,3,2), (3,2,1,1,1),$$
$$ (2,1,1,1,1,1,1), (1,1,1,1,1,1,1,1).$$
The $\tau_{e}(8)$-partitions are:
$$ (7,1), (6,2), (5,3), (5,1,1,1), (4,4), (4,2,1,1), (3,3,1,1), (3,1,1,1,1,1), (1,1,1,1,1,1,1,1),$$
and $\tau_{o}(8)$-partitions are:
$$(8), (6,1,1),(5,2,1), (4,3,1), (4,2,2), (4,1,1,1,1), (3,3,2), (3,2,1,1,1), (2,1,1,1,1,1,1),$$
Indeed $\tau_{e}(8) - \tau_{o}(8) = 0$.\\
\noindent The above theorem can be used to determine the parity of $\tau(n)$. We write down this as a consequence in the corollary below.
\begin{corollary}
For all $n \geq 0$, $\tau(n)$ is odd if and only if $n \equiv 0, 1\pmod{3}$.
\end{corollary}

Darlison Nyirenda \\
{\small School of Mathematics, University of the Witwatersrand}\\
{\small Wits 2050, Johannesburg, South Africa.}\\ 
{\small darlison.nyirenda@wits.ac.za}  \\\\
Abdulaziz M. Alanazi \\
{\small Department of Mathematics, Faculty of Sciences}\\
{\small University of Tabuk, P.O. Box 741, Tabuk 71491, Saudi Arabia}\\
{\small  am.alenezi@ut.edu.sa  }

\end{document}